\documentclass[12pt]{amsart}
\usepackage[foot]{amsaddr}
\usepackage{color}
\usepackage{amsmath,amsfonts}
\usepackage{amsfonts,amsmath,amsthm,amssymb,latexsym}
\usepackage{tikz}

\usepackage{algorithm}
\usepackage{algorithmic}

\newtheorem{theorem}{Theorem}[section]

\newtheorem{lemma}[theorem]{Lemma}
\newtheorem{corollary}[theorem]{Corollary}
\theoremstyle{definition}

\newtheorem{example}[theorem]{Example}

\theoremstyle{remark}

\numberwithin{equation}{section}
\theoremstyle{remark}

\DeclareMathOperator{\lcm}{lcm}
\DeclareMathOperator{\num}{num}
\DeclareMathOperator{\lc}{lc}
\DeclareMathOperator{\lv}{lv}
\DeclareMathOperator{\init}{init}
\DeclareMathOperator{\pinit}{pinit}

\newcommand{\C}{\mathbb{C}}

\newcommand{\PF}{\mathcal{K}}

\newcommand{\sys}{\mathcal{S}}
\newcommand{\yy}{\tilde{y}\,}

\newcommand{\Res}{\mathrm{Res}}

\newcommand\para{\vspace*{2mm}}

\allowdisplaybreaks

\begin{document}

\date{\today}
\title{Algebraic, rational and Puiseux series solutions of systems of autonomous algebraic ODEs of dimension one}

\author{Jos\'{e} Cano}
\address{Dpto. Algebra, an\'alisis Matem\'atico, Geometr\'{\i}a y Topolog\'{\i}a, Universidad de Valladolid, Spain.}
\email{jcano@agt.uva.es}

\author{Sebastian Falkensteiner}
\address{Research Institute for Symbolic Computation (RISC), Johannes Kepler University Linz, Austria.}
\email{falkensteiner@risc.jku.at}

\author{J.Rafael Sendra}
\address{Research Group ASYNACS. Dpto. de F\'{\i}sica y Matem\'aticas, Universidad de Alcal\'a, Madrid, Spain.}
\email{rafael.sendra@uah.es}

\thanks{
First author partially supported by MTM2016-77642-C2-1-P (AEI/FEDER, UE). Second and third authors partially supported by FEDER/Ministerio de Ciencia, Innovaci\'{o}n y Universidades Agencia Estatal de Investigaci\'{o}n/MTM2017-88796-P (Symbolic Computation: new challenges in Algebra and Geometry together with its applications). 
Second author also supported by the Austrian Science Fund (FWF): P 31327-N32. 
Third author is a member of the Research Group ASYNACS (Ref.CT-CE2019/683).
}

\begin{abstract}
In this paper, we study the algebraic, rational and formal Puiseux series solutions of certain type of systems of autonomous ordinary differential equations. 
More precisely, we deal with systems which associated algebraic set is of dimension one. 
We establish a relationship between the solutions of the system and the solutions of an associated first order autonomous ordinary differential equation, 
that we call the reduced differential equation.
Using results on such equations, we prove the convergence of the formal Puiseux series solutions of the system, 
expanded around a finite point or at infinity, and we present an algorithm to describe them. 
In addition, we bound the degree of the possible algebraic and rational solutions, 
and we provide an algorithm to decide their existence and to compute such solutions if they exist.
Moreover, if the reduced differential equation is non trivial, 
for every given point $(x_0,y_0) \in \C^2$, we prove the existence of a convergent Puiseux
series solution $y(x)$ of the original system such that $y(x_0)=y_0$.
\end{abstract}

\maketitle

\noindent \textbf{keywords}
Algebraic autonomous ordinary differential equation,
formal Puiseux series solution,
algebraic solutions,
rational solutions,
convergent solution,
algebraic space curve.

\section{Introduction}

In~\cite{CanoFalkSendra}, we have studied local solutions of first
order autonomous algebraic ordinary differential equations. 
In this paper, we generalize the results obtained there to systems of higher 
order autonomous ordinary differential equations in one unknown function which associated
algebraic set is of dimension one, i.e. the algebraic set is a finite
union of curves and, maybe, points. 
In particular, we prove that every fractional power series solution of such systems is convergent, 
and an algorithm for computing these solutions is provided. 
Note that in~\cite{Denef1984} it is shown that for general
systems of algebraic ordinary differential equations the existence of
non-constant formal power series solutions can not be decided algorithmically.  
Nevertheless, in the case of systems as above, this undecidability
property does not hold.

Finding rational general solutions of such systems has been studied in~\cite{sendra2015rational}.
There, a necessary condition on the degree of the associated algebraic curve is provided.  
If the condition is fulfilled, the solutions are constructed from a
rational parametrization of a birational planar projection of the associated space curve.
Here, we provide an algorithm which decides the existence of not only rational but also algebraic solutions of such systems. 
Differently to the method described in ~\cite{sendra2015rational}, in the current approach we do not need to consider a rational parametrization of the associated curve.
We instead triangularize the given system and we derive from there a
single autonomous ordinary differential equation of first order with
the same non-constant formal Puiseux series solutions. 
We call it the reduced differential equation of the system. 
Since rational or algebraic functions are determined by their Puiseux
series expansion, the reduced differential equation has also the same
algebraic solutions that the original system. 
Furthermore, taking into account that the reduced equation is autonomous and of order one, we bound, 
using the results from~\cite{aroca2005algebraic}, the degree of its possible 
algebraic solutions, and hence of those of the original system.
Once the degree of the solutions is bounded, one may use the algorithm from section 4.3 in~\cite{aroca2005algebraic} to decide the existence and compute such solutions.

We derive the existence and convergence of
formal Puiseux series solutions of such systems (see Theorem~\ref{theorem-existence} and \ref{theorem-convergence}) 
from the corresponding results (see~\cite{CanoFalkSendra}) applied to the obtained reduced differential equation. 
With respect to the convergence of formal solutions a related result is given by Gerasimova and Razmyslov in~\cite{gerasimova2017nonaffine}. 
They show the convergence of formal power series solutions of a system of ordinary differential equations 
under some additional conditions such as that there are no zero-divisors of the differential algebra induced by the system and that the system is of transcendence degree one. 
Their method is based on the fact that the induced differential algebra is finitely generated as an algebra over its base field 
and then they reduce the problem to the Cauchy-Kowalevski theorem. 
Their method does not deal with fractional power series solutions.

In the literature there are several methods to triangularize 
differential systems and to obtain resolvent representations of them, 
see for instance~\cite{CluzeauHubert2003,Ovchinnikov2016} and references therein. 
The description of these methods are quite involved because they apply to general differential systems. 
This paper addresses only ordinary differential systems which 
associated algebraic set has dimension one and we can split the process into an algebraic triangularization part 
and then a straightforward differential elimination process. 
For the algebraic part we use regular chains as described in~\cite{kalk1993} and~\cite{wang2012elimination}.
This simple description of the process allow us
to have a precise relation between the formal Puiseux series solutions
of the original system and those of the reduced differential equation.

The structure of the paper is as follows. 
In Section~\ref{sec-pre} we recall some necessary concepts such as regular chains and regular zeros. 
Section~\ref{sec-systems} is devoted to derive from a system of autonomous ordinary differential equations of dimension one in one unknown function, a finite union of such regular chains. 
From them we derive a single autonomous ordinary differential equation of order one with the same non-constant formal Puiseux series solutions as the original system. 
Using this reduction, the main results in~\cite{CanoFalkSendra} can be generalized to these particular systems (see Theorems~\ref{theorem-convergence} and~\ref{theorem-existence}). 
In Section~\ref{sec-algo} we present an algorithm for this reduction and, using the algorithms in~\cite{aroca2005algebraic} and \cite{CanoFalkSendra}, all formal Puiseux series and algebraic solutions of the original system can be found as we illustrate by examples.

\section{Preliminaries} \label{sec-pre}

We recall the notion of regular chains and regular zeros; for further details we refer to~\cite{kalk1993} and~\cite{wang2012elimination}. 
Let us denote for $f,g \in \C[y_0,\ldots,y_m]$ by $\lv(f)$ the leading variable, by $\lc(f)$ the leading coefficient and by $\init(f)$ the initial of $f$ with respect to the ordering $y_0<\cdots<y_m$. 
In addition, we denote by $\Res_{y_i}(f,g)$ the resultant of $f$ and $g$ with respect to $y_i$. 
Let $\sys=\{F_1,\ldots,F_M\} \subset \C[y_0,\ldots,y_m]$ be a finite system of polynomials in triangular form, i.e. $\lv(F_i)<\lv(F_j)$ for any $1 \leq i<j \leq M$. 
Then we define $\Res(f,\sys)$ as the resultant of $f$ and consecutively $F_M,\ldots,F_1$ with respect to their leading variables, 
i.e. $$\Res(f,\sys)=\Res_{\lv(F_1)}(\cdots \Res_{\lv(F_M)}(f,F_M),\cdots,F_1).$$
Moreover, we define $\init(\sys)=\{\init(F_j)~|~1 \leq j \leq M\}$ and $\pinit(\sys)=\prod_{j=1}^M \init(F_j)$. 
A \textit{regular chain} is a system of algebraic equations $\sys$ in triangular form with the additional property that $\Res(f,\sys) \neq 0$ for any $f \in \init(\sys)$.

Let $K \supseteq \C$ be a field and $\sys \subset \C[y_0,\ldots,y_m]$. 
Then let us denote $$\mathbb{V}_K(\sys)=\{a \in K^{m+1}~|~f(a)=0 \text{ for all } f \in \sys \}.$$
For a regular chain $\sys$, we define a \textit{regular zero} of $\sys$ as an element $a=(a_0,\ldots,a_m) \in \mathbb{V}_{K}(\sys)$ 
such that for $(\sys \cap \C[y_0]) \setminus \C = \emptyset$ the component $a_0$ is transcendental over $\C$ 
and for $1 \leq k \leq m$ with $(\sys \cap \C[y_0,\ldots,y_k]) \setminus \C[y_0,\ldots,y_{k-1}] = \emptyset$ 
the component $a_k$ is transcendental over $\C(a_0,\ldots,a_{k-1})$.

We recall a well-known theorem for the relation between regular chains and regular zeros, see~\cite{wang2012elimination}[Proposition 5.1.5, Corollary 5.1.6].
\begin{theorem} \label{thm-regEquivalent}
Let $\sys=\{F_1,\ldots,F_M\} \subset \C[y_0,\ldots,y_m]$ be a finite system of polynomials in triangular form and denote by $\sys_k$ the first $k$ polynomials of $\sys$. 
Then the following are equivalent:
\begin{enumerate}
	\item $\sys$ is a regular chain.
	\item $|\sys|=1$ or for any $k=2,\ldots,M$ the subsystem $\sys_{k-1}$ is a regular chain and for any regular zero $a$ of $\sys_{k-1}$ and $f \in \sys_{k}$ it holds that $\init(f)(a) \neq 0.$
\end{enumerate}
\end{theorem}

In fact, statement (2) above is used in~\cite{kalk1993} as definition of regular chains. 
Note that in~\cite{yang1990} regular chains are called ''proper ascending chains``, but are defined exactly as in this paper.

Regular chains can be helpful in order to represent algebraic sets as Theorem 5.2.2 in~\cite{wang2012elimination} shows:
\begin{theorem} \label{thm-union}
Let $\sys \subseteq \C[y_0,\ldots,y_m]$ be a polynomial system. 
Then there exists a finite set of regular chains $\sys_1,\ldots,\sys_N \subseteq \C[y_0,\ldots,y_m]$ such that 
\begin{equation} \label{eq-union}
\mathbb{V}_{K}(\sys)= \bigcup_{j=1}^N \mathbb{V}_{K}(\sys_j) \setminus \mathbb{V}_{K}(\pinit(\sys_j)),
\end{equation}
\end{theorem}
We note that, in the notation of~\cite{wang2012elimination}, $\mathbb{V}_{K}(\sys_j / \init(\sys_j))=\mathbb{V}_{K}(\sys_j) \setminus \mathbb{V}_{K}(\pinit(\sys_j))$, as it is also mentioned in chapter 1.5 therein. 
Let us recall that $\mathbb{V}_{\C}(\sys_j)$ is an algebraic set of dimension $m-|\sys_j|$.

There are several implementations for performing computations with regular chains 
and in particular computing regular chain decomposition as in~\eqref{eq-union} such as in the \texttt{Maple}-package \texttt{RegularChains}.

Let $\PF$ be the field of formal Puiseux series expanded around any $x_0 \in \C_{\infty}$, where $\C_{\infty}=\C \cup \{\infty\}$.
We are interested in non-constant formal Puiseux series solutions 
$y(x) \in \PF$ such that $y(x_0)=y_0 \in \C_{\infty}$.
Since the systems we are dealing with, see~\eqref{EQ-AODESystem} below, are invariant under the translation of the independent variable, 
we can assume without loss of generality that the formal Puiseux series is expanded around zero or at infinity such as in~\cite{CanoFalkSendra}.
For any subset 
$\widetilde{\sys} \subseteq \C[y,y',\ldots,y^{(m)}]$ let us denote
$$\mathbb{V}_{{\PF}}(\widetilde{\sys})=\{(a_0,a_1,\ldots,a_m) \in 
{\PF}^{m+1} ~|~ F(a_0,a_1,\ldots,a_m)=0 \text{ for all } F\in \widetilde{\sys} \}.$$

\section{Systems of algebro-geometric dimension one} \label{sec-systems}

Let us consider systems of differential equations of the form
\begin{equation} \label{EQ-AODESystem}
\widetilde{\sys}=\{F_j(y,y',\ldots,y^{(m)})=0\}_{1 \leq j \leq M},
\end{equation}
where $F_1, \ldots, F_M \in \C[y,y',\ldots,y^{(m)}]$ with $m>0$. 
For a field $K \supseteq \C$, by considering $y$ and its derivatives as independent variables, 
we write $\mathbb{V}_{K}(\widetilde{\sys})$ for the algebraic set generated by $\widetilde{\sys}$. 
We assume the dimension of $\mathbb{V}_{\C}(\widetilde{\sys})$ to be one, 
i.e. $\mathbb{V}_{\C}(\widetilde{\sys})$ is a finite union of curves and, maybe, a finite union of points. 
Note that a single AODE of order one can be seen as a system of the type~\eqref{EQ-AODESystem} with $M=m=1$ and is of dimension one.

\begin{lemma} \label{lem:TriangularizedSystem}
For every $\widetilde{\sys}$ as in~\eqref{EQ-AODESystem} we can compute a finite union of regular chains $\sys$ as in~\eqref{EQ-TriangularizedSystem} with the same non-constant formal Puiseux series solutions.
\end{lemma}
\begin{proof}
Let us choose the ordering $y<y'<\cdots<y^{(m)}$. 
By Theorem~\ref{thm-union} there is a regular chain decomposition $\sys_1,\ldots,\sys_N$ 
such that every regular chain has a zeroset of dimension zero or one. 
We omit systems of regular chains starting with an algebraic equation in $y$, since they only lead to constant solutions. 
Thus, the remaining systems are of dimension one and of the form
\begin{equation}\label{EQ-TriangularizedSystem}
\sys=\left\{ \begin{array}{ll}
G_1(y,y')=\sum_{j=0}^{r_1} G_{1,j}(y) \cdot (y')^j=0 \\
G_2(y,y',y'')=\sum_{j=0}^{r_2} G_{2,j}(y,y') \cdot (y'')^j=0 \\
\quad \quad \vdots \\
G_m(y,\ldots,y^{(m)})=\sum_{j=0}^{r_m} G_{m,j}(y,\ldots,y^{(m-1)}) \cdot (y^{(m)})^j=0
\end{array} \right.
\end{equation}
with $r_j \geq 1$ and $\init(G_j)=G_{j,r_j} \neq 0$ for every $1 \leq j \leq m$.

Now we want to study in~\eqref{eq-union} which kind of solutions might
be a solution of a regular chain but not of the original system,
i.e. the solutions of $\sys$ and $\pinit(\sys)=\init(G_1) \cdots \init(G_m)=0$. 
If $y(x)$ is a non-constant formal Puiseux series solution of a $\sys_j$, 
then $(y(x),y'(x),\ldots,y^{(m)}(x))$ is a regular zero of $\sys_j$, 
because $y(x)$ is transcendental over $\C$ 
and for every $1 \leq k \leq m$ we have $$(\sys_j \cap \C[y,\ldots,y^{(k)}]) \setminus \C[y,\ldots,y^{(k-1)}]=G_k \neq \emptyset.$$
Then, by Theorem~\ref{thm-regEquivalent}, $$\init(G_2)(y(x),y'(x)),\ldots,\init(G_m)(y(x),\ldots,y^{(m-1)}(x)) \neq 0.$$ 
Since $\init(G_1)(y)=0$ is an algebraic equation in $y$, 
there can only be constant common zeros of $\sys_j$ and $\pinit(\sys_j)$.
\end{proof}

System~\eqref{EQ-TriangularizedSystem} could be further decomposed into systems with the factors of $G_1$ as initial equations.
However, for our purposes it is sufficient that $G_1$ and its separant, namely $\frac{\partial\, G_1}{\partial u_1}$, have no common differential solutions, i.e. if $G_1(y(x),y'(x))=0$ for $y(x) \in \PF$ then $\frac{\partial\, G_1}{\partial u_1}(y(x),y'(x)) \neq 0$.
To ensure this we consider $G_1 \in \C[u_0,u_1]$ to be square-free and with no factor in $\C[u_0]$ or $\C[u_1]$; compare with the hypotheses in~\cite{CanoFalkSendra}.

Moreover, we can assume without loss of generality for every solution $y(x) \in \PF$ 
of a system $\sys$ as in~\eqref{EQ-TriangularizedSystem} that $y(0)=y_0\in \C$.
Otherwise, if $y_0=\infty$, consider the change of variable $y=1/\yy$.
Let $G_j^*(\yy,\yy',\ldots,\yy^{(j)})$ be the numerator of
$G_j(1/\yy,(1/\yy)',\ldots,(1/\yy)^{(j)})$, and let $\sys^*$ be the
system $\{ G_j^*=0\}_{1\leq j \leq m}$.
In this situation, if $y(x) \in \PF$ is a solution of $\sys$ such that $y(x_0)=\infty$, then $\yy(x)=1/y(x)$ is a formal Puiseux series solution of $\sys^*$ with $\yy(x_0)\in \C$.
Moreover, for $j>0$, the $j$-th derivative of $\yy$ can be written as
\[ \yy^{(j)}=\dfrac{-y^{j-1}\,y^{(j)}+P_j(y,\ldots,y^{(j-1)})}{y^{j+1}} \]
where $P_j\in \C[u_0,\ldots,u_{j-1}]$.
In this situation, we consider the rational map
$$\Phi:\mathbb{C}^{m+1}\setminus \mathbb{V}(u_0) \rightarrow \mathbb{C}^{m+1} \setminus \mathbb{V}(w_0); (u_0,\ldots,u_m)\mapsto (w_0,\ldots,w_m),$$ where $w_0=1/u_0$ and
\[ w_j= \dfrac{-u_{0}^{j-1}\,u_j+P_j(u_0,\ldots,u_{j-1})}{u_{0}^{j+1}}. \]
Since the equality above is linear in $u_j$, $\Phi$ is birational.
In addition, taking into account that $u_0$ is not a factor of $G_1(u_0,u_1)$, one has that the Zariski closure of $\Phi(\mathbb{V}_{\C}(\sys))$ is $\mathbb{V}_{\C}(\sys^*)$.
Since $\dim(\mathbb{V}_{\C}(\sys))=1$, also $\dim(\mathbb{V}_{\C}(\sys^*))=1$ and one may proceed with $\sys^*$ instead of $\sys$.

\para

For a given system $\sys$ as in~\eqref{EQ-TriangularizedSystem} we now associate 
a finite set of bivariate polynomials $\mathcal{H}(\sys)=\{H_1,\ldots,H_m\} \subset \C[u_0,u_1]$.
According to \cite{ritt1950differential}[page 6], for every $j \geq 2$ there exists a differential polynomial $R_{j-1}$ of order $j-1$ such that
\begin{equation}\label{eq-ritt}
G_1^{(j-1)}(y,\ldots,y^{(j)})=\frac{\partial\, G_1}{\partial u_1}(y,y') \cdot y^{(j)} + R_{j-1}(y,\ldots,y^{(j-1)}).
\end{equation}
Then, for $2 \leq j \leq m$, we introduce the rational functions
\begin{equation} \label{EQ:help3}
A_j(u_0,\ldots,u_{j-1})=\dfrac{-R_{j-1}(u_0,\ldots,u_{j-1})}{\frac{\partial\, G_1}{\partial u_1}(u_0,u_1)}.
\end{equation}
Now we recursively substitute in~\eqref{EQ:help3} the variables $u_2,\ldots,u_m$ by\\ $A_2(u_0,u_1),\ldots,A_m(u_0,\ldots,u_{m-1})$ to obtain the new rational functions $B_2,\ldots,B_m \in \C(u_0,u_1)$:
\begin{equation}\label{eq-B}
\left\{
\begin{array}{l}
B_2(u_0,u_1)=A_2(u_0,u_1) \\
B_3(u_0,u_1)=A_3(u_0,u_1,B_2(u_0,u_1))\\
\,\,\,\,\,\,\vdots \\
B_m(u_0,u_1)=A_m(u_0,u_1,B_2(u_0,u_1),\ldots,B_{m-1}(u_0,u_1)).
\end{array}
\right.
\end{equation}
Observe that the denominators of the rational functions $A_j$ are powers of the separant and depend only on $u_0$ and $u_1$.
Finally we set
\begin{equation}\label{eq-H}
\left\{
\begin{array}{l}
H_1(u_0,u_1)=\num(G_1(u_0,u_1)) \\
H_2(u_0,u_1)=\num(G_2(u_0,u_1,B_2(u_0,u_1))) \\
\,\,\,\vdots \\
H_m(u_0,u_1)=\num(G_m(u_0,u_1,B_2(u_0,u_1),\ldots,B_m(u_0,u_1))),
\end{array}
\right.
\end{equation}
where $\num(f)$ denotes the numerator of the rational function $f$.

In this situation, we introduce a new autonomous first order algebraic differential equation, namely
\begin{equation} \label{eq-gcd}
H(y,y')=\gcd(H_1,\ldots,H_m)(y,y')=0,
\end{equation}
and call it the \textit{reduced differential equation} (of $\sys$).
Note that by construction, $H$ divides $G_1$. 
Moreover, if $\sys$ contains only one single equation $G_1$, then the reduced differential equation of $\sys$ is equal to $G_1$.

\begin{theorem}\label{theorem:sys-gcd}
Let $\sys$ be as in~\eqref{EQ-TriangularizedSystem}.
Then $\sys$ and its reduced differential equation have the same non-constant formal Puiseux series solutions.
\end{theorem}
\begin{proof}
Let $G_1$ be the square-free starting equation of $\sys$.
First of all we observe that $\gcd(G_1,\frac{\partial\, G_1}{\partial u_1})=1$ (see e.g. proof of Theorem 4.4. in~\cite{FalkensteinerSendra_2018}).
Therefore, if $y(x)\in \PF$ is non-constant and $G_1(y(x),y'(x))=0$, then $\frac{\partial\, G_1}{\partial u_1}(y(x),y'(x))\neq 0$.

Let $y(x)\in \PF$ be a non-constant formal Puiseux solution of $\sys$.
Then\\ $G_1(y(x),y'(x))=0$, and hence $G_1^{(j-1)}(y(x),\ldots,y^{(j)}(x))=0$.
Applying formula~\eqref{eq-ritt} and since $\frac{\partial\, G_1}{\partial u_1}(y(x),y'(x))\neq 0$, we get that
\[ A_j(y(x),\ldots,y^{(j-1)}(x))=y^{(j)}(x) \]
and we obtain that
\[ B_j(y(x),y'(x))=y^{(j)}(x). \]
Therefore,
\begin{align*} H_j(y(x),y'(x))&=\num(G_j(y(x),y'(x),B_2(y(x),y'(x)),\ldots,B_j(y(x),y'(x))) \\ &=G_j(y(x),\ldots,y^{(r)}(x))=0
\end{align*}
for every $2 \leq j \leq m$. 
By B\'{e}zout's identity, there exists $Q \in \C[y]$ such that the polynomial $Q \cdot \gcd(H_1,\ldots,H_m)$ is an algebraic combination of the $H_j$. 
Since the equation $Q(y)=0$ has only constant solutions, $y(x)$ is a solution of $H(y,y')=0$.

Conversely, let $y(x) \in \PF$ be a non-constant solution of the reduced differential equation $H(y,y')=\gcd(H_1,\ldots,H_m)(y,y')=0$.
Then, $G_1(y(x),y'(x))=H_1(y(x),y'(x))=0$ and as observed above, we obtain $\frac{\partial\, G_1}{\partial u_1}(y(x),y'(x))\neq 0$.
Thus, for every $1 \leq j \leq m$ the denominator of $G_j(u_0,u_1,B_2(u_0,u_1),\ldots,B_j(u_0,u_1))$ does not vanish at $(y(x),y'(x))$.
Taking into account \eqref{eq-H}, it follows that
$$G_j(y(x),y'(x),B_2(y(x),y'(x)),\ldots,B_j(y(x),y'(x)))=0.$$
Now, let us recursively show $B_j(y(x),y'(x))=y^{(j)}(x)$ for $2 \leq j \leq m$, which proves the theorem:
Since $B_2(y(x),y'(x))=A_2(y(x),y'(x))$ and by \eqref{EQ:help3}, $$R_1(y(x),y'(x))=-B_2(y(x),y'(x))\, \frac{\partial\, G_1}{\partial u_1}(y(x),y'(x)).$$
Then, by \eqref{eq-ritt},
\[ 0=G_1^{(1)}(y(x),y'(x),y''(x))=\frac{\partial\, G_1}{\partial u_1}(y(x),y'(x)) (y''(x)-B_2(y(x),y'(x))). \]
Using that $\frac{\partial G_1}{\partial u_1}(y(x),y'(x))\neq 0$, we obtain $B_2(y(x),y'(x))=y''(x)$.
Now, let us assume that $B_i(y(x),y'(x))=y^{(i)}(x)$ for $2 \leq i \leq j$.
By~\eqref{eq-B}, $B_{j+1}(y(x),y'(x))=A_{j+1}(y(x),\ldots,y^{(j)}(x))$.
Then, reasoning as above, we obtain
\begin{align*}
0&=G_1^{(j)}(y(x),\ldots,y^{(j)}(x))\\&=\frac{\partial\, G_1}{\partial u_1}(y(x),y'(x)) (y^{(j+1)}(x)-B_{j+1}(y(x),y'(x))),
\end{align*}
and hence, $B_{j+1}(y(x),y'(x))=y^{(j+1)}(x).$
\end{proof}

\begin{corollary} \label{cor:trivialReducedEquation}
Let $\sys$ be as in~\eqref{EQ-TriangularizedSystem} and let its reduced differential equation $H$ be a product of factors in $\C[y]$ and $\C[y']$.
Then $\sys$ has only linear formal Puisex series solutions, i.e. solutions of the form $\alpha\,x+\beta$ for some $\alpha,\beta \in \C$.
\end{corollary}
\begin{proof}
From the construction of the reduced differential equations, and the assumption that $G_1,\ldots,G_m \neq 0$, we know that $H \neq 0$.
For every factor in $\C[y]$ of $H$, there are only constant solutions, and for every factor in $\C[y']$, there are only linear solutions of $H(y,y')=0$.
Then from Theorem~\ref{theorem:sys-gcd} the statement follows.
\end{proof}

Let $\widetilde{\sys}$ be as in~\eqref{EQ-AODESystem}.
Then, by Lemma~\ref{lem:TriangularizedSystem}, it can be written as the union of systems $\sys_1,\ldots,\sys_K$ of the form~\eqref{EQ-TriangularizedSystem}.
Let $H_1,\ldots,H_K$ be the reduced differential equations of these systems $\sys_1,\ldots,\sys_K$.
Then, as a consequence of Theorem~\ref{theorem:sys-gcd}, $\widetilde{\sys}$ and
\begin{equation} \label{eq-sysprod}
H(y,y')=\lcm(H_1,\ldots,H_K)(y,y')=0
\end{equation}
have the same non-constant formal Puiseux series solutions.
Equation~\eqref{eq-sysprod} is called a reduced differential equation of $\widetilde{\sys}$.
Now we are in the position to generalize the theoretical results obtained in~\cite{CanoFalkSendra}, in particular the two main Theorems in section 3, to systems of dimension one.

\begin{theorem} \label{theorem-convergence}
All formal Puiseux series solutions of the system of differential equations~\eqref{EQ-AODESystem}, expanded around a finite point or at infinity, are convergent.
\end{theorem}
\begin{proof}
By Theorem~\ref{theorem:sys-gcd}, the system~\eqref{EQ-AODESystem} and its reduced differential equation have the same non-constant solutions. 
By~\cite{CanoFalkSendra}[Theorem 10] all formal Puiseux series solutions of $H=0$ are convergent. 
Since constant solutions are convergent as well, the statement follows.
\end{proof}

\begin{theorem} \label{theorem-existence}
Let $\widetilde{\sys}$, as in~\eqref{EQ-AODESystem}, have a non-linear solution.
Then for any point $(x_0,y_0) \in \C^2$ there exists an analytic solution $y(x)$ of~\eqref{EQ-AODESystem} such that $y(x_0)=y_0$.
\end{theorem}\label{theorem-degree-bound-alg-sol}
\begin{proof}
By Corollary~\ref{cor:trivialReducedEquation}, the reduced differential equation of $\widetilde{\sys}$ has at least one irreducible factor depending on $y$ and $y'$.
Then by~\cite{CanoFalkSendra}[Theorem 11] the statement follows.
\end{proof}

\begin{theorem}
Let $y(x)$ be a non-constant formal Puiseux series solution of $\widetilde{\sys}$ as in \eqref{EQ-AODESystem} algebraic over $\C(x)$ and $H$ the reduced differential equation of $\widetilde{\sys}$. 
Then the minimal polynomial $G(x,Y)$ of $y(x)$ fulfills the degree bounds 
\begin{equation} \label{deg-bound}
\deg_{x}(G)\leq\deg_{y'}(H),\quad\text{and}\quad \deg_{Y}(G)\leq \deg_{y}(H)+\deg_{y'}(H).
\end{equation}
In particular if $y(x)$ is a rational solution of $\widetilde{\sys}$,
then its degree, the maximum of the degree of the numerator and denominator, 
is less than or equal to $\deg_{y'}(H)$.
\end{theorem}
\begin{proof}
By Theorem~\ref{theorem:sys-gcd}, $y(x)$ is a solution of the autonomous first order differential equation $H(y,y')=0$. 
Then by Theorems 3.4 and 3.8 in~\cite{aroca2005algebraic} the degree bounds~\eqref{deg-bound} follow.
\end{proof}

\section{Algorithms and Examples} \label{sec-algo}

In this section we outline an algorithm based on the results in Section~\ref{sec-systems} 
to derive the reduced differential equation from a given system ~\eqref{EQ-AODESystem}. 
By \textsf{Triangularize}$(\widetilde{\sys})$ we refer to the computation of a regular chain decomposition of a given system $\widetilde{\sys}$ as in Theorem~\ref{thm-union}.
Then, using the algorithms in~\cite{CanoFalkSendra}, 
it is possible to algorithmically describe all formal Puiseux series solutions of the given system. 
We illustrate this in the subsequent examples.

\begin{algorithm}[H]
\caption{ReduceSystem} \label{ALG:Reduce}
\begin{algorithmic}
	\REQUIRE A finite system of autonomous algebraic ordinary differential equations $\widetilde{\sys} \subset \C[y,\ldots,y^{(m)}]$ which associated algebraic set is of dimension one.
	\ENSURE The reduced differential equation of $\widetilde{\sys}$.
	\STATE Set $\mathfrak{S}=$\textsf{Triangularize}$(\widetilde{\sys})$ and $H=1$.
	\FOR{every $\sys \in \mathfrak{S}$ of dimension one let $G_1$ be the polynomial in $\C[u_0,u_1]$ associated to the equation of $\sys$ depending on $y,y'$}
	\STATE Take the square-free part of $G_1(u_0,u_1)$ and divide by its factors in $\C[u_0]$ and $\C[u_1]$; call it $G_1^*(u_0,u_1)$.
	\STATE Replace in $\sys$ the equation $G_1(y,y')=0$ by $G_1^*(y,y')$; call it $\sys^*$.
	\STATE Compute the associated set $\mathcal{H}(\sys^*)=\{H_1,\ldots,H_m\}$ as in~\eqref{eq-ritt}-\eqref{eq-H} and set $H=\lcm(H,\gcd(H_1,H_2,\ldots,H_m))$.
	\ENDFOR \\
	\RETURN $H$.
\end{algorithmic}
\end{algorithm}

\begin{example} \label{ex-1-part1}
Let us consider the system of differential equations given by
\begin{equation} \label{ex-1-General}
	\widetilde{\sys}=
	\left\{ \begin{array}{ll}
	yy'y''+y'^3-yy''-y'^2=0 \\
	yy'-1-y'^2-yy''=0.
	\end{array} \right.
\end{equation}
The system $\widetilde{\sys}$ can be decomposed into the system of regular chains 
\begin{equation*}
	\sys_1=
	\left\{ \begin{array}{ll}
	G_1=yy'-1=0 \\
	G_2=y'^2+yy''=0
	\end{array} \right.
\quad \text{ and } \quad 
	\sys_2=
	\left\{ \begin{array}{ll}
	y'-1=0 \\
	2-y+yy''=0
	\end{array} \right.
\end{equation*}
For the system $\sys_1$ the starting equation $G_1$ is already square-free with no factor in $\C[y]$ or $\C[y']$ and we set $H_1=G_1$.
By computing $\frac{d\,G_1}{dx}(y,y',y'')$ and setting it to zero we obtain $y''=\frac{-y'^2}{y}.$
Hence, $$H_2(y,y')=\num\left(G_2\left(y,y',\frac{-y'^2}{y}\right)\right)=H_1(y,y').$$
Then the reduced differential equation of $\sys$ is $$H(y,y')=\gcd(H_1,H_2)(y,y')=yy'-1=0.$$

For the system $\sys_2$ we obtain $H_1(y,y')=y'-1$ and $H_2(y,y')=2-y'$, which are coprime. 
Hence, the reduced differential equation of $\sys_2$ is equal to one and therefore, the reduced differential equation of $\widetilde{\sys}$ is $H(y,y')=yy'-1$.

We remark that by using the \texttt{RosenfeldGroebner}-command from \texttt{Maple}, 
which uses regular differential chains as described in~\cite{CluzeauHubert2003}, 
the reduced differential equation $H(y,y')=0$ of $\widetilde{\sys}$ can be found as well.
\end{example}

The next algorithm describes all formal Puiseux series solutions of a system $\widetilde{\sys}$ which associated algebraic set is of dimension one. 
We use algorithm \textsf{PuiseuxSolve} described in~\cite{CanoFalkSendra} whose input is an autonomous ordinary differential equation of order one and algorithm \textsf{ReduceSystem} from above. 
The output is a finite set of truncations in one-to-one correspondence to all Puiseux series solutions.

\begin{algorithm}[H]
\caption{PuiseuxSolveSystem} \label{ALG:PuiseuxSolveSystem}
\begin{algorithmic}
	\REQUIRE A finite system of autonomous algebraic ordinary differential equations $\widetilde{\sys} \subset \C[y,\ldots,y^{(m)}]$ which associated algebraic set is of dimension one.
	\ENSURE A set $\Sigma$ of all solution truncations of $\widetilde{\sys}$ such that the truncation can be uniquely extended. 
	$\Sigma$ has non constant solutions if and only if $\Sigma\neq \emptyset$.
	\STATE Set $\Sigma=\emptyset$ and $H(y,y')=$\textsf{ReduceSystem}$(\widetilde{\sys})$.
        \STATE Let $H^{*}(y,y')$ be the polynomial obtained after factoring out factors in $\C[y]$ and $\C[y']$ 
        and taking the square-free part of $H(y,y')$.
        \STATE If $H^{*}$ is not a constant, then set $\Sigma=$\textsf{PuiseuxSolve}$(H^{*})$.
        \STATE Add to $\Sigma$ the non constant linear solutions of $\widetilde{\sys}$. 
        This can be done by making the ansatz $y(x)=\alpha\,x+\beta$ with unknown $\alpha$ and $\beta$ 
        and plug it into the equations and solving the algebraic system obtained in $\alpha$ and $\beta$.   
	\RETURN $\Sigma$.
\end{algorithmic}
\end{algorithm}

We can devise a similar algorithm to compute a set of truncations of
Puiseux solutions expanded at the infinity point replacing in the
above algorithm the algorithm \textsf{PuiseuxSolve} by the algorihtm 
\textsf{PuiseuxSolveInfinity} from~\cite{CanoFalkSendra}. 
However, the uniqueness of the
extension can in general not be ensured.

\para

The next algorithm decides if a system $\widetilde{\sys}$ as 
\eqref{EQ-AODESystem} has an
algebraic solution and compute some of them in the affirmative case. 
Its correctness is based on the proof of 
Theorem~\ref{theorem-degree-bound-alg-sol} where it is shown 
that the non constant algebraic solutions of the system
$\widetilde{\sys}$ are the non constant algebraic solutions of the reduced equation
of $\widetilde{\sys}$. In section 4 of~\cite{aroca2005algebraic} there
is a description of an algorithm that decides if an autonomous
differential equation $H_i(y,y')=0$ has algebraic solutions and
compute them in the affirmative case. This algorithm needs that the
polynomials $H_i(y,y')$ are irreducible. Hence, the next algorithm is
not factor free. 
Let us call the output of this algorithm \textsf{AlgSol}$(H_i)$.

\begin{algorithm}[H]
\caption{AlgSolutionSystem} \label{ALG:AlgSolutionSystem}
\begin{algorithmic}
	\REQUIRE A finite system of autonomous algebraic ordinary
        differential equations 
  $\widetilde{\sys} \subset \C[y,\ldots,y^{(m)}]$ which associated algebraic set is of dimension one.
	\ENSURE A set $\Sigma$ of algebraic solutions of
        $\widetilde{\sys}$ or the emptyset such that system
        $\widetilde{\sys}$ has an algebraic solution if and only if
        $\Sigma\neq\emptyset$.
	\STATE Set $\Sigma=\emptyset$ and  
        $H(y,y')=$\textsf{ReduceSystem}$(\widetilde{\sys})$.
        \STATE Let $H^{*}(y,y')$ be the polynomial obtained after
        factoring out factors in $\C[y]$ and $\C[y']$.
        \FOR{ each irreducible factor $H_i(y,y')$ of $H^{*}(y,y')$ }
        \STATE Add to $\Sigma$ the output of the above mentioned algorithm \textsf{AlgSol}$(H_i)$.
        \ENDFOR
        \STATE Add to $\Sigma$ the non constant linear solutions of
        $\widetilde{\sys}$.   
	\RETURN $\Sigma$.
\end{algorithmic}
\end{algorithm}


\begin{example}\label{ex-1-part2}
Let us consider system \eqref{ex-1-General} of Example \ref{ex-1-part1}. 
By Theorem \ref{theorem:sys-gcd}, the solutions are those of the reduced differential equation $$H(y,y')=y\,y'-1=0.$$
We obtain all the formal Puiseux series solutions, expanded around zero, by the one-parameter family of solutions
\begin{align*}
y(x)=y_0+\frac{x}{y_0}-\frac{x^2}{2y_0^3}+\frac{x^3}{2y_0^5}+\mathcal{O}(x^4)
\end{align*}
with $y_0 \in \C \setminus \{0\}$, and the particular solutions 
$$y(x)= \pm \sqrt{2}\, x^{1/2}.$$
There is no formal Puiseux series solution with the initial value $y(0)=\infty$.
The only linear solutions of $\widetilde{\sys}$ are $y(x)=\pm 1$.
The possible algebraic solutions $y(x)$ have a minimal polynomial
$G(x,Y)$ with degree bound of $\deg_xG\leq \deg_{y'}(H)=1$ and
$\deg_YG\leq \deg_{y'}(H)+\deg_{y}(H)=2$.
They are given by the zeros of
$$G_{y_0}(x,Y)=Y^2-2\,(x+\frac{y_0^2}{2}).$$
\end{example}

The assumption on the dimension of the given system is crucial in our work. 
Otherwise for instance Theorem~\ref{theorem-convergence} does not hold anymore as the following example shows.

\begin{example} \label{ex-non convergent}
Let us consider $F(x,y,y')=x^2\,y'-y+x$. 
The non-convergent formal power series $$y(x)=\sum_{j \geq 0} j!\,x^{j+1}$$ is a solution of $F=0$. 
Hence, $y(x)$ is also a zero of $\frac{d\,F}{dx}(x,y,y',y'')$ and consequently, of the resultant of $F$ and $\frac{d\,F}{dx}$ with respect to $x$, namely $$\Res_x\left(F,\frac{d\,F}{dx}\right)=y''+y''^2y^2-y''y'+4y''y'y-y'^2-2y''y'^2y-4y'^3y+y'^4.$$
Note that $\{ \Res_x\left(F,\frac{d\,F}{dx}\right)(y,y',y'')=0 \}$ defines a system of autonomous ordinary algebraic differential equations of algebro-geometric dimension two.
\end{example}

\appendix

\bibliographystyle{authordate1}

\begin{thebibliography}{30}
\bibitem{aroca2005algebraic}
J.M.~Aroca, J.~Cano, R.~Feng, and X.-S.~Gao.
\newblock Algebraic general solutions of algebraic ordinary differential equations.
\newblock In {\em Proceedings of the 2005 International Symposium on Symbolic and Algebraic Computation}, 29--36. ACM Press, 2005.

\bibitem{CanoFalkSendra}
J.~Cano and S.~Falkensteiner and J.R.~Sendra.
\newblock Existence and convergence of {Puiseux} series solutions for first order autonomous differential equations.
\newblock {\em pre-print}, abs/1908.09196, 2019.

\bibitem{CluzeauHubert2003}
T.~Cluzeau and E.~Hubert.
\newblock Resolvent Representation for Regular Differential Ideals.
\newblock {\em Applicable Algebra in Engineering, Communication and Computing}, 13(5):395--425, 2003.

\bibitem{Denef1984}
J.~Denef and L.~Lipshitz.
\newblock Power Series Solutions of Algebraic Differential Equations.
\newblock {\em Mathematische Annalen}, 267:213--238, 1984.
	
\bibitem{FalkensteinerSendra_2018}
S.~Falkensteiner and J.R.~Sendra.
\newblock Formal Power Series Solutions of First Order Autonomous Algebraic Ordinary Differential Equations.
\newblock {\em Mathematics in Computer Science}, doi.org/10.1007/s11786-019-00431-6, 2019.

\bibitem{gerasimova2017nonaffine}
O.V.~Gerasimova and Y.P.~Razmyslov.
\newblock Nonaffine differential-algebraic curves do not exist.
\newblock {\em Moscow University Mathematics Bulletin} 72, 3:89--93 2017.

\bibitem{kalk1993}
M.~Kalkbrener.
\newblock A Generalized Euclidean Algorithm for Computing Triangular Representations of Algebraic Varieties.
\newblock {\em Journal of Symbolic Computation}, 15:143--167, 1993.

\bibitem{sendra2015rational}
A.~Lastra and J.R.~Sendra and L.X.C.~Ng\^{o} and F.~Winkler.
\newblock Rational general solutions of systems of autonomous ordinary differential equations of algebro-geometric dimension one.
\newblock {\em Publicationes Mathematicae Debrecen}, 86(1--2):49--69, 2015.

\bibitem{Ovchinnikov2016}
A.~Ovchinnikov and G.~Pogudin and N.~Thieu Vo.
\newblock Bounds for elimination of unknowns in systems of differential-algebraic equations.
\newblock {\em pre-print}, abs/1610.04022, 2016.

\bibitem{ritt1950differential}
J.F.~Ritt.
\newblock {\em Differential algebra}, vol. 33.
\newblock American Mathematical Society, 1950.

\bibitem{wang2012elimination}
D.~Wang
\newblock {\em Elimination methods}.
\newblock Springer Science \& Business Media, 2012.

\bibitem{yang1990}
L.~Yang and J.~Zhang.
\newblock Searching dependency between algebraic equations: an algorithm applied to automated reasoning.
\newblock {\em Oxford University Express, Inc.}, 147--156, 1994.

\end{thebibliography}

\end{document}